\documentclass[12pt,a4paper,oneside]{amsart}
\usepackage[margin=1.2in]{geometry}
\usepackage{amsfonts, amsmath, amssymb, amsthm, mathtools}
\usepackage[colorlinks=true,citecolor=blue]{hyperref}
\usepackage[utopia]{mathdesign}
\usepackage[mathscr]{euscript}
\usepackage[utf8]{inputenc}

\usepackage{graphicx}
\usepackage{parskip}
\usepackage{enumerate}
\usepackage{verbatim}
\usepackage{tikz}
\usepackage{units}

\newtheorem{theorem}{Theorem}[section]
\newtheorem{lemma}[theorem]{Lemma}
\newtheorem{conjecture}[theorem]{Conjecture}

\newtheorem{claim}[theorem]{Claim}
\newtheorem{probs}[theorem]{Problem}
\usepackage{booktabs}

\theoremstyle{definition}
\newtheorem{definition}[theorem]{Definition}

\newtheorem{remark}[theorem]{Remark}

\usepackage{enumerate}

\numberwithin{equation}{section}

\DeclareMathOperator{\dist}{dist}
\DeclareMathOperator{\conv}{conv}

\renewcommand{\epsilon}{\varepsilon}
\renewcommand{\phi}{\varphi}
\renewcommand{\kappa}{\varkappa}
\newcommand{\df}[1]{\emph{\color{black!75!blue}#1}}
\newcommand{\R}{\mathbb R}

\newcommand{\mrd}{M_{r,d}}
\newcommand{\xm}{X_{[m]}}

\begin{document}
\title{Topology of geometric joins}
\author{Imre~B\'ar\'any$^\star$}

\address{\linebreak Imre~B\'{a}r\'{a}ny \hfill \hfill \linebreak
Alfr\'{e}d R\'{e}nyi Mathematical Institute, \hfill \hfill \linebreak
 Hungarian Academy of Sciences, \hfill \hfill \linebreak
P.O. Box 127, 1364 Budapest,  Hungary \hfill \hfill \linebreak
and \linebreak
Department of Mathematics, \hfill \hfill \linebreak
University College London, \hfill \hfill \linebreak
Gower Street, London WC1E 6BT, UK}

\email{barany.imre@renyi.mta.hu}

\thanks{{$^\star$}Partially supported by ERC
  Advanced Research Grant no 267165 (DISCONV), and by Hungarian
  National Research Grant K 83767.} 

\author{Andreas~F.~Holmsen$^\diamond$}
\address{\linebreak Andreas~F.~Holmsen \hfill \hfill \linebreak  
Department of Mathematical Sciences, \hfill \hfill\linebreak
KAIST,\hfill \linebreak 
291 Daehak-ro, Daejeon 305-701, South Korea} 

\email{andreash@kaist.edu}

\thanks{{$^\diamond$}Supported  by Basic Science Research Program through the
National Research Foundation of Korea funded by the Ministry of
Education, Science and Technology (NRF-2010-0021048).}

\author{Roman~Karasev$^{\dagger}$}
\address{\linebreak Roman Karasev \hfill \hfill \linebreak
Dept. of Mathematics, \hfill\hfill\linebreak
Moscow Institute of Physics and Technology, \hfill\hfill \linebreak
Institutskiy per. 9, Dolgoprudny, Russia 141700 \hfill\hfill \linebreak
and \linebreak
Institute for Information Transmission Problems RAS, \hfill\hfill \linebreak
Bolshoy Karetny per. 19, \hfill\hfill \linebreak
Moscow, Russia 127994}

\email{r\_n\_karasev@mail.ru}
\urladdr{http://www.rkarasev.ru/en/}

\thanks{{$^{\dagger}$}Supported by the Dynasty foundation.}

\begin{abstract}
We consider the geometric join of a family of subsets of the Euclidean space. This is a construction frequently used in the (colorful) Carath\'eodory and Tverberg theorems, and their relatives. We conjecture that when the family has at least $d+1$ sets, where $d$ is the dimension of the space, then the geometric join is contractible. We are able to prove this when $d$ equals $2$ and $3$, while for larger $d$ we show that the geometric join is contractible provided the number of sets is quadratic in $d$. We also consider a matroid generalization of geometric joins and provide similar bounds in this case.
\end{abstract}

\maketitle

\section{Introduction}

The purpose of this paper is to introduce the notion of a geometric join, and to study its topological connectedness. The geometric join is a natural object which appears in the proof of the colorful Carath\'eodory theorem~\cite{bar1982} and Tverberg's theorem~\cite{tver1966}; see chapter 8 in \cite{mat2002} for a detailed explanation. Recently, it was also shown in \cite{hol2014} that the colorful version of Hadwiger's transversal theorem \cite{aro2009} is closely related to the connectedness of the geometric join.

\begin{definition}
Let $X_1, \dots, X_m$ be subsets of the Euclidean space $\R^d$. The \df{geometric join} of $X_1$, $\dots$, $X_m$ is the set of all convex combinations $t_1x_1 + \cdots   + t_mx_m \in \R^d$ where $x_i \in X_i$, $t_i\ge 0$, and $\sum\limits_{^{i=1}}^{_m} t_i = 1$. The geometric join of the subsets $X_1, \dots, X_m$ of $\R^d$ is denoted by $\xm$.
\end{definition} 

\begin{remark}
  In this paper we will consider the case when the subsets $X_1, \dots ,X_m$ are finite, but our results can easily be extended to the case when the $X_i$ are arbitrary compact subsets of $\R^d$. 
\end{remark}

The subsets $X_1$, $\dots$, $X_m$ are often referred to as \df{color classes}, and a subset $Y\subset X_1$ $\cup$ $\cdots$ $\cup X_m$ is called \df{colorful} if $|Y\cap X_i|\leq 1$ for every $i$. The convex hull of a colorful subset is called a \df{colorful simplex}. In other words, the geometric join $\xm$ is the union of all colorful simplices spanned by $X_1\cup \cdots \cup X_m$.  

\medskip

Let us start by pointing out some simple examples. Consider a point set $X \subset \R^d$. Carath{\'e}odory's theorem \cite{Car1907} states that for any point $p$ in the convex hull of $X$, i.e. $x\in \conv X$, there exists a subset $Y\subset X$ such that $|Y|\leq d+1$ and $p\in \conv Y$. This means that if we color each point in $X$ by  $d+1$ distinct colors, then every subset $Y\subset X$ with $|Y|\leq d+1$ spans a colorful simplex, so in this case the geometric join is the same as $\conv X$. Using our notation, this means if $X = X_1$ $= \cdots$ $= X_{d+1}$, then $X_{[d+1]} = \conv X$. A well-known generalization of Carath{\'e}odory's theorem is the colorful Carath{\'e}odory theorem due to B{\'a}r{\'a}ny \cite{bar1982} which states that if $X_1$, $\dots$, $X_{d+1}$ are subsets of $\R^d$ and a point $p$ is contained in $\conv X_i$ for all $1\leq i \leq d+1$, then the point $p$ is contained in a colorful simplex spanned by $X_1$, $\dots$, $X_{d+1}$. Equivalently, this can be stated as:  $\bigcap_{i=1}^{d+1} \conv X_i \subset \xm$. In fact the stronger statement $\bigcap_{i\neq j} \conv (X_i\cup X_j) \subset \xm$ also holds \cite{aro2009,hol2008}.

\medskip

Here we consider the following problem.

\begin{probs} \label{join-conn-prob} 
Give sufficient conditions in terms of $m$ and $d$ for the contractibility or $k$-connectedness of $\xm$. 
\end{probs}

We may compare the geometric join with the \emph{abstract join} (see~\cite{mat2003}), which can be regarded as a geometric join after putting the $X_i$'s into $\mathbb R^D$, with sufficiently large $D$, so that the affine hulls of the $X_i$'s are in general position. It is known that an abstract join of $m$ finite sets, each of cardinality greater than one, is homotopic to a wedge of $(m-1)$-dimensional spheres. The geometric join can be regarded as a piecewise linear image of the abstract join in the ambient space $\mathbb R^d$, where $X_i$'s reside, and its homotopy type may be different from that of the abstract join.  

In the subsequent sections we show that the geometric join has certain connectivity for sufficiently large $m$, depending on $d$. These are partial results towards establishing the following. 

\begin{conjecture} \label{conj-main} 
The geometric join $\xm$ is contractible whenever $m\geq d+1$.  
\end{conjecture}

This conjecture is open even when $m=d+1$ and each $X_i$ consists of two elements. If any of the $X_i$ were a singleton, then $X_{[m]}$ is trivially contractible. Actually, the authors do not agree whether the conjecture should be true or false, and perhaps it is better to look for a counterexample. 

\medskip

In section~\ref{star-sect} we show that $\xm$ is starshaped whenever $m>d(d+1)$. This is a simple consequence of Tverberg's theorem \cite{tver1966}, and of course implies contractibility of $\xm$.   

Section~\ref{morse-dist-sec} introduces a technique for studying the homotopy type of compact subsets of $\R^d$ via an analogue of the Morse theory to the distance function. As a consequence we can show that $\xm$ is $(k-2)$-connected whenever $m>\frac{dk}{2}$. This is done in section \ref{join-sec}. Note that this implies that $\xm$ is contractible whenever $m>\frac{d(d+1)}{2}$, which is a slight improvement on the approach of section \ref{star-sect}. 

In section~\ref{simple} we apply the nerve theorem to show that the geometric join in $\mathbb{R}^d$ is simply connected whenever $m>\frac{d+2}{2}$, and it is easily seen that this bound is best possible. This implies that our Conjecture holds when $d=2$. It should be noted that the case $d=2$ was previously verified in \cite{bois1991,tot2010}. 

Section~\ref{dim2-3-sec} gives a proof of our Conjecture for $d = 3$. Our proof uses some basic observations about geometric joins in $\R^2$ together with the ``strong'' colorful Carath\'{e}odory theorem \cite{aro2009,hol2008}. 

In section \ref{matroid} we generalize the notion of geometric joins by replacing the color classes by an arbitrary matroid. This gives rise to a generalization of our main Problem, and it turns out that many of our methods also work in this more general setting. 

\section{Starshapedness of the geometric join} \label{star-sect}

We start with an observation that the geometric join is {\em starshaped} for sufficiently large $m$, which obviously implies contractibility.  

\begin{theorem} \label{join-starshaped}
If $m > d(d+1)$, then $\xm$ is starshaped.
\end{theorem}

\begin{proof}
For each $1\leq i \leq m$, choose one element $x_i\in X_i$, and let $T = \{x_1\dots,x_m\}$. By Tverberg's theorem~\cite{tver1966} there is a partition $T = T_1 \cup\cdots\cup T_{d+1}$ and a point $t\in \R^d$ such that $t \in \conv T_j$ for each $j$. We will show that every point $x \in \xm$ can be ``seen'' from $t$. 

It suffices to consider the case when $x$ belongs to the boundary of $\xm$, so by Carath\'eodory's theorem we may assume $x$ belongs to a colorful simplex of dimension at most $d-1$. Thus $x$ is contained in the convex hull of a colorful subset $Y$ with $|Y| \leq d$. By the pigeon-hole principle there exists some $T_j$ such that $T_j\cup Y$ is a colorful subset. Therefore the closed segment $[t,x]$ is contained in $\conv(T_j\cup Y)$ which is contained in $\xm$.   
\end{proof}

\begin{remark}
In the previous argument the Tverberg point can be replaced by any point in the $d$-core of $X = \bigcup X_i$, which is the intersection of the convex hulls of all sets $X\setminus Y$ where $|Y| = d$. Here in fact, it suffices to take all sets $X\setminus Y$ where $Y$ is a colorful subset with $|Y|=d$. 
\end{remark}

\begin{remark}
Actually, Krasnoselskii's theorem (that a compact $C$ is starshaped iff every $d+1$ of its points are seen from some point of $C$, see~\cite{dgk1963}) implies that $\xm$ is starshaped when $m \geq (d+1)^2+1$. 
\end{remark}

\begin{remark}
It should be noted that in $\R^2$ the geometric join $\xm$ is starshaped for $m\geq 3$ \cite{bois1991,tot2010}, but in $\R^3$ the geometric join $X_{[4]}$ is not necessarily starshaped as was shown in \cite{tot2010}.
\end{remark}

\section{Topology of subsets of $\mathbb R^d$ through the 
distance function} \label{morse-dist-sec}

Suppose we have a compact set $S\subset \mathbb R^d$ and we want to study its homotopy type. One possible way to do this (see also~\cite{eh2010}, where such methods are widely discussed) is to apply an analogue of the Morse theory to the distance function 
\[\rho_S : \mathbb R^d \to \mathbb R,\] 
which is $\rho_S(x) = \dist(x, S)$. The sets $$ S(t) = \{x\in \mathbb R^d : \rho_S(x) \le t\} $$ in this case are just $t$-neighborhoods of $S$. For $t=0$ the set $S(0)$ is equal to $S$ and $S(t) \subset S(t')$ whenever $t\leq t'$. 

If the function $\rho_S$ (which we simply denote by $\rho$) were a smooth function (which can only happen in the trivial case of convex $S$) we could study the problem using the ordinary Morse theory, however, in general the differential $d\rho$ is not always defined. For a given $x_0\not\in S$ let $P(x_0)$ denote the set of points in $S$ which are closest to $x_0$. If $P(x_0)$ consists of a single point then (for reasonable sets $S$) the differential $d\rho$ is the unit normal in the direction of $x_0-P(x_0)$. Otherwise the function $\rho$ has no differential at $x_0$.

The next informal observation is the following: If $\conv P(x_0)$ does not contain $x_0$ then varying $t$ in a neighborhood of $t_0 = \rho(x_0)$ does not influence the topology of $S(t)$ near $x_0$. This is because $S(t_0)$ in the first order approximation is the complement of the convex cone 
\[\{ x\in \mathbb R^d : \forall y\in P(x_0) (x - x_0, x_0 - y) \ge 0\},\] 
which has nonempty interior; and $S(t)$ in a neighborhood of $x_0$ looks similarly when $t$ is close to $t_0$. 

Of course, the above argument is informal and we want to give a 
rigorous proof in the following useful case.

\begin{theorem} \label{morse-dist-contr} 
Let $S$ be the union of a finite number of compact convex sets in $\R^d$. If for any $x_0\not\in S$ we have $x_0\not\in \conv P(x_0)$, then $S$ is contractible.
\end{theorem} 

\begin{remark} 
In~\cite{pan2001}, under the same condition $x_0\not\in \conv P(x_0)$, it was proved that $\mathbb R^d\setminus X$ is contractible for sets $X$ of another kind. The methods of~\cite{pan2001} also seem to imply Theorem~\ref{morse-dist-contr} but we provide a different proof here, which is short and self-contained. 
\end{remark}

\begin{proof} 
Let $S = \bigcup_{i=1}^m C_i$ where $F= \{C_1, \dots, C_m\}$ is a family of compact convex sets in $\R^d$. By the nerve theorem (see~\cite{bor1948} or~\cite[Corollary~4G.3]{hat2002}) the homotopy type of $S$ is determined by the nerve of $F$, which we denote by $N$. This is the abstract simplicial complex with vertex set $[m]$ where $\sigma\subset [m]$ is a simplex of $N$ if and only if $\bigcap_{i\in\sigma}C_{i} \neq \emptyset$. To be more precise, the nerve theorem implies that $S$ is homotopy equivalent to the geometric realization of $N$. For $t\geq 0$, let $N(t)$ denote the nerve of the family $F(t) = \{C_1(t), \dots, C_m(t)\}$. The idea is to show that the homotopy type of the nerve $N(t)$ does not change as the parameter $t$ increases. This will prove the claim of the theorem, since $N(t)$ is an $(m-1)$-dimensional simplex for all sufficiently large $t$. Alternatively, by thinking of this process in reverse, we show that $N$ can be obtained from the $(m-1)$-dimensional simplex by a sequence of simplicial collapses.

We will prove the theorem for the case when the members of $F$ are smooth, strictly convex, and in some suitable ``general position'', which will be explained below. Any other configuration can be reduced to this case by approximating every body in the family by a smooth and strictly convex body, maintaining the other general position assumptions, such that the nerve of the approximating family remains the same. Since the nerve lemma applies for every such approximating family and the nerve remains the same, we conclude that the contractibility of the union of the approximating family implies the contractibility of the union of the original family.

Now consider the situation when the nerve $N(t)$ changes. This means there is some subfamily $G(t) \subset F(t)$ which is intersecting for all $t \geq t_0>0$,  but not intersecting for any $t<t_0$. Here we impose the additional ``general position'' assumption, namely, that this is the only change in $N(t)$ which happens for all $t$ sufficiently close to $t_0$. Since the members of $F$ are smooth and strictly convex it follows that the members of $G(t_0)$ intersect in a unique point $x_0 \in \R^d$. Clearly we have $\rho(x_0) \leq t_0$, and we claim that the inequality must be strict. Define the sets \[I = \{i \in [m] \: : \: \dist(x_0, C_i) < t_0\} \: \mbox{ and } \: J = \{j \in [m] \: : \: \dist(x_0, C_j)  = t_0\}.\] If $\rho(x_0) = t_0$, then $I=\emptyset$ and the set $P(x_0)$ contains a unique point $y_j\in C_j$ for every $j\in J$. Since the sets $C_j(t_0)$ have a single point, $x_0$, in common, the vectors $y_j-x_0$ contain the origin in their convex hull and therefore $x_0$ is contained in $\conv P(x_0)$. This contradicts the hypothesis. Therefore we may assume that $\rho(x_0) < t_0$, which implies $I\neq\emptyset$, and therefore $I \cup J$ is a partition of $[m]$.

The point $x_0$ is contained in the interior of the set $C_i(t_0)$ for every $i\in I$, while it is on the boundary of the set $C_j(t_0)$ for every $j\in J$. Let $\hat{C}_j(t_0)$ be the set obtained from $C_j(t_0)$ by cutting off, by a hyperplane, a small cap centered at $x_0$, for every $j\in J$. Since the bodies $C_j(t_0)$ are strictly convex, these small caps can be chosen arbitrarily close to $x_0$, and since we are cutting off by hyperplanes, the bodies $\hat{C}_j(t_0)$ remain convex so the nerve theorem still applies. The resulting family $\{\hat{C}_j(t_0)\}_{j\in J}$ will have empty intersection, and if the removed caps are chosen sufficiently small, we will have \[\bigcup_{i\in [m]} C_i(t_0) = \left( \bigcup_{i\in I} C_i(t_0) \right) \: \cup \: \bigcup_{j\in J}\hat{C}_j(t_0).\] Let $\hat{N}(t_0)$ denote the nerve of the family $\{C_i(t_0)\}_{i \in I} \cup \{\hat{C}_j(t_0)\}_{j\in J}$. By the nerve theorem, $N(t_0)$ and $\hat{N}(t_0)$ are homotopy equivalent, and clearly $\hat{N}(t_0) = N(t)$ for all $t < t_0$ which are sufficiently close to $t_0$. 
\end{proof}

For a weaker conclusion than contractibility, we have the following. 

\begin{theorem} \label{morse-dist-conn} 
Let $S$ be the union of a finite number of compact convex sets in $\R^d$. If for any $x_0\not\in S$ we have $x_0\notin \conv Y$ where $Y\subset P(x_0)$ with $|Y|\le k$, then $S$ is $(k-2)$-connected.  
\end{theorem}

\begin{proof} 
The proof of Theorem \ref{morse-dist-contr} goes through as before, using one additional observation: Each time the nerve $N(t)$ changes, as $t$ increases, either there is no change in the homotopy type, or the new simplex being added has at least $k+1$ vertices, so in both cases $\pi_i(N(t))$ is preserved for $i\le k-2$. 
\end{proof}

\section{Contractibility of the geometric join} \label{join-sec}

We now apply the methods from the previous section to our main Problem concerning the connectivity of $\xm$. Our first result is the following. 

\begin{theorem} \label{contr-dsq} 
If $m > \frac{d(d+1)}{2}$, then $\xm$ is contractible. 
\end{theorem}

\begin{proof} 
The results of Section~\ref{morse-dist-sec} can be applied to $\xm$ since it is the union of the colorful simplices spanned by $X_1\cup \cdots \cup X_m$. We assume that $\xm$ is not contractible and obtain an upper bound on $m$, that is, the number of distinct colors. From Theorem~\ref{morse-dist-contr}, there is a point $x_0\not\in \xm$ that is contained in the convex hull of the set of its closest points $P(x_0)$. By Carath\'eodory's theorem there is a set $\{y_1,\ldots,y_k\}  \subset P(x_0)$ with $k \le d+1$ such that $x_0 \in \conv \{y_1,\ldots,y_k\}$ and each $y_i \in \conv S_i$ where $S_i$ is some colorful subset with $|S_i|\leq d$. Therefore $|S_1\cup \cdots \cup S_{k}| \leq d(d+1)$. Now we will show that each color appears in $S_1\cup \cdots \cup S_k$ at least twice. Assuming the contrary, we have to consider two cases:   

\begin{itemize}

\item {\em Some color $j$ is not used.} Define open halfspaces 
\[H_i^+ = \{x\in \mathbb R^d : \langle x - y_i, x_0- y_i\rangle > 0\}.\] 
These open halfspaces cover $\mathbb R^d$ because $x_0\in \conv\{y_1, \ldots, y_k\}$ and every $H_i^+$ is disjoint from its respective $S_i$. A point $p$ of color $j$, which exists by our assumption, is contained in some $H_i^+$, and therefore the segment $[y_i, p]$ is closer to $x_0$ than $y_i$. This is a contradiction since $S_i \cup \{p\}$ is a colorful subset whose convex hull contains the segment $[y_i, p]$.

\item {\em Some color $j$ is used only once}. Let $p$ denote the unique point in $S_1\cup \cdots \cup S_k$ of color $j$. Again, $p \in H_i^+$ for some $i$, and the set $S_i$ does not contain the color $j$ since $p$ was the only point of this color. Therefore $S_i\cup \{p\}$ is a colorful subset whose convex hull is closer to $x_0$ than the $y_i$'s, and again we obtain a contradiction. 
\end{itemize}

Therefore there are at most $\frac{d(d+1)}{2}$ distinct colors. \end{proof}

Similarly we prove:

\begin{theorem} \label{conn-dsq} 
If $m > \frac{dk}{2}$, then $\xm$ is $(k-2)$-connected. \end{theorem}

\begin{proof} The previous proof goes through as before, but we need only consider subsets $\{y_1,\ldots, y_\ell\}\subseteq P(x_0)$ consisting of at most $k$ points in view of Theorem~\ref{morse-dist-conn}. The rest of the proof is the same.
\end{proof}

\begin{remark} 
If the points in $X_1\cup\dots\cup X_m$ are in appropriate general position, then the sum of codimensions of $S_i$ must be at least $d+1$, otherwise perturbations will destroy the inclusion $x_0\in\conv \{y_1,\ldots, y_\ell\}$. It follows that the number of vertices of the $S_i$ will be at most $(k-1)(d+1)$ and the inequality in Theorem~\ref{conn-dsq} can be relaxed to $m > \frac{(k-1)(d+1)}{2}$ in this case. 
\end{remark}

\section{An improved bound for simple connectedness} 
\label{simple} 

We have a feeling that our application of Morse theory for the distance function is not the optimal approach for attacking our main Problem. To illustrate this we give an improved sufficient condition for simple connectedness. 

\begin{theorem} \label{simp-conn}
If $m>\frac{d+2}{2}$, then $\xm$ is simply connected.
\end{theorem}

\begin{proof} 
We think of the geometric join as a map $f: Y \to \mathbb{R}^d$ where $Y = Y_1 *\cdots * Y_m$ is the abstract join and $f$ is linear on every simplex of $Y$. Thus, $X_i = f(Y_i)$, $\xm = f(Y)$, and for $m\geq 2$, $\xm$ and $Y$ are connected. The case $d=1$ is obvious, so we suppose $d\geq 2$.

By the nerve theorem, a path in $Y$ from $p$ to $q$ can be regarded as a sequence $(\sigma_1, \dots ,\sigma_k)$ where the $\sigma_i \in Y$ are $(m-1)$-dimensional simplices, $p\in\sigma_1$, $q\in\sigma_k$, and $\sigma_i \cap \sigma_{i+1} \neq \emptyset$ for every $1\leq i <k$. Likewise, a path in $\xm$ from $f(p)$ to $f(q)$ can be regarded as a sequence $(f(\sigma_1), \dots , f(\sigma_k))$ where the $\sigma_i \in Y$ are $(m-1)$-simplices, $p\in \sigma_1$, $q\in\sigma_k$, and $f(\sigma_i)\cap f(\sigma_{i+1}) \neq \emptyset$ for every $1\leq i <k$. 

Now consider a path $(f(\sigma_1),\ldots, f(\sigma_k))$ in $\xm$. Suppose there are consecutive simplices $\sigma_i$ and $\sigma_{i+1}$ such that $\sigma_i\cap\sigma_{i+1} = \emptyset$, but $f(\sigma_i)\cap f(\sigma_{i+1}) \neq \emptyset$. Since $m>\frac{d+2}{2}$, there is a proper face $\tau$ of either $\sigma_i$ or $\sigma_{i+1}$ such that $f(\sigma_i)\cap f(\tau)\cap f(\sigma_{i+1}) \neq \emptyset$. Therefore $\tau$ is contained in an $(m-1)$-simplex $\tau_{i} \in Y$ such that $\sigma_i \cap \tau_i \neq\emptyset$, $\tau_i\cap \sigma_{i+1}\neq \emptyset$, and $f(\sigma_i) \cap  f(\tau_i) \cap f(\sigma_{i+1}) \neq \emptyset$. The nerve theorem implies that the paths, $(f(\sigma_1$, $\dots$, $f(\sigma_i), f(\sigma_{i+1})$, $\dots$, $f(\sigma_k))$ and $(f(\sigma_1)$, $\dots$, $f(\sigma_i), f(\tau_i), f(\sigma_{i+1})$, $\dots$, $f(\sigma_k))$ 
are homotopic in $\xm$. For each consecutive pair $f(\sigma_i)$, $f(\sigma_{i+1})$ such that $\sigma_i\cap\sigma_{i+1} = \emptyset$, this procedure can be repeated, thereby removing all such pairs. Therefore for any element $\gamma \in \pi_1(X)$ there is a $\gamma'\in \pi_1(Y)$ such that $\gamma$ and $f(\gamma')$ are homotopic. Since $m\geq 3$ we have $\pi_1(Y) = 0$ which completes the proof. 
\end{proof}

\begin{remark} The inequality $m>\frac{d+2}{2}$ is tight which can be seen by the following example. Let $d=2k$ and $Y = Y_1* \cdots *Y_{k+1}$ where $|Y_i|=2$. Then $Y \cong S^k$. We can map $Y$ into $\mathbb{R}^{d}$ so that a single pair of opposite $k$-simplices of $Y$ intersect in an interior point, resulting in a space homeomorphic to a $k$-sphere with a single pair of antipodal points identified. Such a space is not simply connected.
\end{remark}

\section{Dimensions 2 and 3} \label{dim2-3-sec}

It was established in \cite{bois1991}, and independently in \cite{tot2010}, that for $d=2$ the geometric join $X_{[3]}$ is starshaped, which implies that Conjecture \ref{conj-main} holds for $d=2$. It is easily seen that their arguments extend to $m > 3$. Here we establish the next case of our Conjecture. 

\begin{theorem}\label{conj:d3}
If $d=3$ and $m\geq 4$, then $\xm$ is contractible.  
\end{theorem}

The proof of Theorem \ref{conj:d3} is based on two observations. The first one is the ``strong'' colorful Carath{\'e}odory theorem, established independently in \cite{aro2009} and \cite{hol2008}.

\begin{lemma}\label{obs:colorcara}
Let $m\geq d+1$ and suppose that the origin is not contained in $\xm$. Then there exists $1\leq i < j \leq m$ and an affine hyperplane that strictly separates $X_i\cup X_j$ from the origin.  
\end{lemma} 

\begin{remark}
In fact there exists a subset $I\subset [m]$ with $|I|\geq m-d+1$ and an affine hyperplane which strictly separates $\cup_{i\in I}X_i$ from the origin, but for our purpose we only need Lemma \ref{obs:colorcara} as stated.    
\end{remark}

The second observation extends the fact that the geometric join in the plane is starshaped. Let $X_1$ and $X_2$ be finite sets in $\R^2$, and consider their geometric join $X_{[2]}$. The complement $\mathbb{R}^2\setminus X_{[2]}$ is a collection of open regions, one of which is unbounded. Let $\tilde{X}$ denote the complement of the (unique) unbounded region. Obviously $\tilde{X}$ is a compact region, but the following also holds.

\begin{claim}\label{obs:starshape}
For finite sets $X_1$ and $X_2$ in $\mathbb{R}^2$, the set $\tilde{X}$ is starshaped. 
\end{claim}

\begin{proof} The claim is obvious if $|X_1|=1$ or $|X_2|=1$. The geometric join $X_{[2]}$ can be regarded as a drawing of a complete bipartite graph in the plane where edges are drawn as straight segments. Direct each edge so that it goes from a vertex $v\in X_1$ to a vertex $w\in X_2$. Then we can assign a unique angle $\alpha \in [0,\pi]$ to each pair of edges. Choose a pair of edges $e_1, e_2$ which  
\begin{itemize}
\item have a point in common (which may or may not be a vertex), and
\item maximize the angle $\alpha_0$. 
\end{itemize}
We consider the case where $e_1$ and $e_2$ intersect in interior points and $\alpha_0<\pi$. (The cases when $e_1$ and $e_2$ intersect in a vertex or when $\alpha_0=\pi$ are treated similarly.) Suppose $e_1 = [v_1, w_1]$, $e_2 = [v_2, w_2]$, $v_i\in X_1$, $w_i\in X_2$, and $x = e_1\cap e_2$. We show that any edge $e = [v,w]$, where $v\in X_1$ and $w\in X_2$, is visible from $x$ within a bounded region enclosed by some cycle of $X_{^{[2]}}^2$. This will prove that $\tilde{X}$ is starshaped from $x$. By the maximality of the angle $\alpha_0$ the points of $X_1\cup X_2$ must be contained in the closed antipodal sectors bounded by $\angle v_1xv_2$ and $\angle w_1xw_2$. Up to symmetries there are two cases to consider. 
\begin{enumerate}
\item $e$ is contained in the closed sector bounded by $\angle v_1xv_2$, and the line containing $e$ intersects the ray from $x$ through $v_1$. 
\item $e$ crosses the closed sector bounded by $\angle v_1xw_2$.
\end{enumerate}
Case (1) splits into two subcases: (a) The ray from $v$ through $w$
intersects the ray from $x$ through $v_1$. Then the edge $e$ is
visible from $x$ within the cycle $v w v_1 w_1$. See figure below
(left). (b) The ray from $w$ through $v$ intersects the ray from $x$
through $v_1$. Then the edge $e$ is visible from $x$ within the cycle
$v w v_2 w_2$. See figure below (right). 

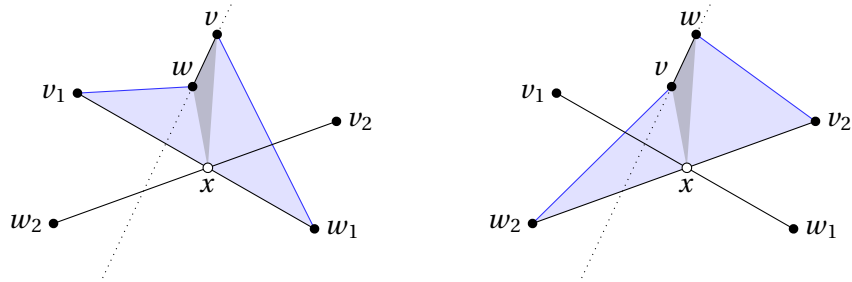
\begin{center}
\begin{tikzpicture}
  \begin{scope}[scale=.9]
    \begin{scope}[rotate = 20]
      \begin{scope}
        \fill[blue, opacity = .12] (.2,1.2) -- (.8,1.8) -- (310:1.8cm) -- 
        (130:2.2cm) -- cycle; 
        \draw[blue, opacity =.8] (.2,1.2) -- (130:2.2cm) (.8,1.8) --
        (310:1.8cm); 
        \fill[gray, opacity =.4] (0,0) -- (.2,1.2) -- (.8,1.8) -- cycle;
      \end{scope}
      \begin{scope}
        \fill (0:2cm) circle [radius = .07];  
        \fill (180:2.4cm) circle [radius = .07];  
        \fill (130:2.2cm) circle [radius = .07];  
        \fill (310:1.8cm) circle [radius = .07];  
        \draw (0:2cm) -- (180:2.4cm);
        \draw (130:2.2cm) -- (310:1.8cm);
        \fill[white] (0,0) circle [radius = .07];
        \draw (0,0) circle [radius = .07];
      \end{scope}
      \begin{scope}
        \draw[dotted] (-2,-1) --(.2,1.2)  (.8,1.8) -- (1.2,2.2);
        \fill (.2,1.2) circle [radius = .07];
        \fill (.8,1.8) circle [radius = .07];
        \draw (.2,1.2) -- (.8,1.8);
      \end{scope}
      \begin{scope}
        \node [below] at (0,0) {\small $x$};
        \node [left] at (180:2.4cm) {\small $w_2$};
        \node [left] at (130:2.2cm) {\small $v_1$};
        \node [right] at (0:2cm) {\small $v_2$};
        \node [right] at (310:1.8cm) {\small $w_1$};
        \node at (0.15,1.5) {\small $w$};
        \node at (.8,2.05) {\small $v$};
      \end{scope}
    \end{scope}
\begin{scope}[xshift = 7cm]
  \begin{scope}[rotate = 20]
    \begin{scope}
      \fill[blue, opacity = .12] (.2,1.2) -- (.8,1.8) -- (0:2cm) --
      (180:2.4cm) -- cycle;  
      \draw[blue, opacity = .8] (.2,1.2) -- (180:2.4cm) (.8,1.8) -- (0:2cm);
      \fill[gray, opacity =.4] (0,0) -- (.2,1.2) -- (.8,1.8) -- cycle;
    \end{scope}
    \begin{scope}
      \fill (0:2cm) circle [radius = .07];  
      \fill (180:2.4cm) circle [radius = .07];  
      \fill (130:2.2cm) circle [radius = .07];  
      \fill (310:1.8cm) circle [radius = .07];  
      \draw (0:2cm) -- (180:2.4cm);
      \draw (130:2.2cm) -- (310:1.8cm);
      \fill[white] (0,0) circle [radius = .07];
      \draw (0,0) circle [radius = .07];
    \end{scope}
    \begin{scope}
      \draw[dotted] (-2,-1) --(.2,1.2)  (.8,1.8) -- (1.2,2.2);
      \fill (.2,1.2) circle [radius = .07];
      \fill (.8,1.8) circle [radius = .07];
      \draw (.2,1.2) -- (.8,1.8);
    \end{scope}
    \begin{scope}
      \node [below] at (0,0) {\small $x$};
      \node [left] at (180:2.4cm) {\small $w_2$};
      \node [left] at (130:2.2cm) {\small $v_1$};
      \node [right] at (0:2cm) {\small $v_2$};
      \node [right] at (310:1.8cm) {\small $w_1$};
      \node at (0.15,1.5) {\small $v$};
      \node at (.8,2.05) {\small $w$};
    \end{scope}
  \end{scope}
\end{scope}
\end{scope}
\end{tikzpicture}
\end{center}

Case (2) splits into two subcases: (a) $v$ is contained in the closed sector bounded by $\angle v_1xv_2$. See figure below (left). (b) $w$ is contained in the closed sector bounded by $\angle v_1xv_2$. See figure below (right). In both cases the edge $e$ is visible from $x$ within the cycle $v w v_2 w_1$. 
\end{proof}

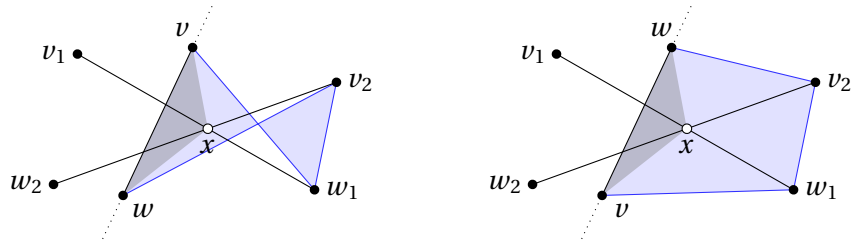
\begin{center}
\begin{tikzpicture}
  \begin{scope}[scale=.9]
    \begin{scope}[rotate = 20]
      \begin{scope}
        \fill[blue, opacity = .12] (.2,1.2) -- (-1.5,-.5) -- (0:2cm) -- 
        (310:1.8cm) -- cycle; 
        \draw[blue, opacity =.8] (-1.5,-.5) -- (0:2cm) -- (310:1.8cm)
        -- (.2,1.2);
        \fill[gray, opacity =.4] (0,0) -- (-1.5,-.5) -- (.2,1.2) -- cycle;
      \end{scope}
      \begin{scope}
        \fill (0:2cm) circle [radius = .07];  
        \fill (180:2.4cm) circle [radius = .07];  
        \fill (130:2.2cm) circle [radius = .07];  
        \fill (310:1.8cm) circle [radius = .07];  
        \draw (0:2cm) -- (180:2.4cm);
        \draw (130:2.2cm) -- (310:1.8cm);
        \fill[white] (0,0) circle [radius = .07];
        \draw (0,0) circle [radius = .07];
      \end{scope}
      \begin{scope}
        \draw[dotted] (-2,-1) --(-1.5,-.5)  (.2,1.2) -- (.7,1.7);
        \fill (.2,1.2) circle [radius = .07];
        \fill (-1.5,-.5) circle [radius = .07];
        \draw (-1.5,-.5) -- (.2,1.2);
      \end{scope}
      \begin{scope}
        \node [below] at (0,0) {\small $x$};
        \node [left] at (180:2.4cm) {\small $w_2$};
        \node [left] at (130:2.2cm) {\small $v_1$};
        \node [right] at (0:2cm) {\small $v_2$};
        \node [right] at (310:1.8cm) {\small $w_1$};
        \node at (0.15,1.5) {\small $v$};
        \node at (-1.3,-.8) {\small $w$};
      \end{scope}
    \end{scope}
\begin{scope}[xshift = 7cm]
  \begin{scope}[rotate = 20]
      \begin{scope}
        \fill[blue, opacity = .12] (.2,1.2) -- (-1.5,-.5)  -- 
        (310:1.8cm) -- (0:2cm) -- cycle; 
        \draw[blue, opacity =.8] (-1.5,-.5)  -- (310:1.8cm) -- (0:2cm)
        -- (.2,1.2);
        \fill[gray, opacity =.4] (0,0) -- (-1.5,-.5) -- (.2,1.2) -- cycle;
      \end{scope}
      \begin{scope}
        \fill (0:2cm) circle [radius = .07];  
        \fill (180:2.4cm) circle [radius = .07];  
        \fill (130:2.2cm) circle [radius = .07];  
        \fill (310:1.8cm) circle [radius = .07];  
        \draw (0:2cm) -- (180:2.4cm);
        \draw (130:2.2cm) -- (310:1.8cm);
        \fill[white] (0,0) circle [radius = .07];
        \draw (0,0) circle [radius = .07];
      \end{scope}
      \begin{scope}
        \draw[dotted] (-2,-1) --(-1.5,-.5)  (.2,1.2) -- (.7,1.7);
        \fill (.2,1.2) circle [radius = .07];
        \fill (-1.5,-.5) circle [radius = .07];
        \draw (-1.5,-.5) -- (.2,1.2);
      \end{scope}
      \begin{scope}
        \node [below] at (0,0) {\small $x$};
        \node [left] at (180:2.4cm) {\small $w_2$};
        \node [left] at (130:2.2cm) {\small $v_1$};
        \node [right] at (0:2cm) {\small $v_2$};
        \node [right] at (310:1.8cm) {\small $w_1$};
        \node at (0.15,1.5) {\small $w$};
        \node at (-1.3,-.8) {\small $v$};
      \end{scope}
  \end{scope}
\end{scope}
\end{scope}
\end{tikzpicture}
\end{center}

Theorem \ref{conj:d3} will be deduced from the following slightly stronger claim. 

\begin{claim} \label{claim:r3}
Let $m\geq 4$ and suppose the origin is not contained in $\xm$. Then there exists an infinite ray $R\subset \R^3$ from the origin such that $\xm \cap R = \emptyset$.    
\end{claim}

\begin{proof}
We may suppose the $X_i$'s are on the unit sphere centered at the origin and argue using spherical convexity. Suppose $X_1$ and $X_2$ are the sets found in Claim \ref{obs:colorcara}, so they are contained in some open hemisphere. We may therefore regard the geometric join of $X_1$ and $X_2$ as a planar geometric join, and let $\tilde{X}$ denote the starshaped set from Claim \ref{obs:starshape}.  

Notice that the boundary of $\tilde{X}$ cannot separate the set $(-X_{3})\cup\cdots\cup (-X_m)$. If there exists points $p,q \in (-X_{3})\cup\cdots\cup (-X_m)$ such that $p\in \tilde{X}$ and $q\notin\tilde{X}$, then there also exists $v\in X_{i} $ and $w\in X_j$, $3\leq i<j\leq m$, such that $-v$ and $-w$ are separated by the boundary of $\tilde{X}$. In this case the geodesic connecting $-v$ to $-w$ intersects the boundary of $\tilde{X}$, and a boundary segment of $\tilde{X}$ is made up of a geodesic connecting $a\in X_1$ and $b\in X_2$, which implies that the simplex spanned by $a, b, v, w$ contains the origin, contradicting the assumption that the origin is not contained in $X$.  

We may therefore assume that $p\cap \tilde{X}=\emptyset$ for every $p\in -(X_{3})\cup\cdots\cup (-X_m)$. If not, we reverse the argument and define $\tilde{X}$ for any pair $X_i$ and $X_j$ with $3\leq i < j \leq m$.

Let $c$ be the center of $\tilde{X}$. We show that $-c$ is the direction we are looking for. Suppose the contrary, that $-c$ is contained in some triangle $x_ix_jx_k$ spanned by points from distinct color classes $X_i$,  $X_j$, and $X_k$, respectively. This implies that the origin is contained in the simplex $cx_ix_jx_k$. Consider the following cases:  

\begin{enumerate}
\item If $i=1$, $j=2$, and $k>2$, then $-x_k$ is contained in the triangle $cx_1x_2$. This is a contradiction since the triangle $cx_1x_2$ is contained in $\tilde{X}$. 

\item If $i=1$ and $k>j>2$, then the geodesic connecting $-x_j$ and $-x_k$ intersects the geodesic connecting $c$ and $x_1$. But this implies that the geodesic connecting $-x_j$ and $-x_k$ intersects the boundary of $\tilde{X}$, which cannot happen (which was explained two paragraphs above). 

\item If $k>j>i>2$, then consider a point $v\in X_1$. Either the simplex $vx_ix_jx_k$ contains the origin, or $-c$ is covered by a triangle involving the vertex $v$, which puts us in case (2) above. 
\end{enumerate} 
\end{proof}

We are now in position to prove Theorem \ref{conj:d3}.

\begin{proof}
Since $m\geq d=3$, Theorem \ref{simp-conn} implies that $\xm$ is simply connected, and Claim \ref{claim:r3} implies that the second homology group of $\xm$ vanishes. It follows that $X$ is contractible. 
\end{proof}

\section{Geometric joins of matroids} \label{matroid}

Kalai and Meshulam \cite{kalmesh} showed that the color classes in the colorful Helly theorem can be replaced by an arbitrary matroid. A similar generalization was given for the ``strong'' colorful Carath\'{e}odory theorem in \cite{hol-int} and for the Colorful Hadwiger transversal theorem in \cite{hol2014}. The purpose of this last section is to show that the notion of geometric joins can be generalized in the same way, and that most of our methods from the previous chapters work in this more general setting.

Let us recall that a matroid $M$ on a finite set $E$ can be defined as a non-empty family of subsets of $E$ called the \df{independent sets} which satisfy the following properties:

\begin{itemize}
\item If $B$ is independent and $A\subset B$, then $A$ is independent.
\item If $A$ and $B$ are independent and $|A| < |B|$, then there exists an element $b\in B\setminus A$ such that $A\cup \{b\}$ is independent. 
\end{itemize}

The second condition is often called the {\em independence augmentation axiom} for matroids. We will assume that the union of all independent sets equals the ground set $E$, which is the same as restricting ourselves to matroids which are loopless. For a subset $S\subset E$ the rank of $S$, denoted by $\mbox{rk}(S)$, is the maximum cardinality of an independent set contained in $S$, and the rank of the matroid equals $\mbox{rk}(E)$. Notice that the independent sets of a matroid form an abstract simplicial complex which is often referred to as the \df{independence complex} of the matroid.

\begin{definition}
Let $E\subset \R^d$ be a finite set and let $M$ be a matroid defined on $E$. The \df{geometric join} of $M$ is the set of all convex combinations $t_1x_1 + \cdots + t_kx_k\subset \R^d$ where $\{x_1, \dots x_k\}$ is independent in $M$. The geometric join of a matroid $M$ of rank $r$ defined on a finite set of points in $\R^d$ is denoted by $\mrd$.
\end{definition}

Our previous definition of geometric join is obtained by noticing that the colorful subsets of a family of finite sets $X_1, \dots, X_m$ form the independent sets of a matroid. In the general case of a matroid, the convex hull of an independent set will be called an \df{independent simplex}. In other words, $\mrd$ is the union of all independent simplices of $M$.

As before, we can think of the geometric join of a matroid as a piecewise linear image of the independence complex into the ambient space $\R^d$ where the ground set of the matroid resides. It is a well-known fact that the independence complex of a matroid of rank $r$ is $(r-2)$-connected (see for instance \cite{BjKoLo}), but the homotopy type of $\mrd$ might be different from that of its independence complex. 

It should be clear from the discussion above that our main Problem can be studied in the setting of arbitrary matroids.   

\begin{probs}
\label{mat-join-conn-prob} Give sufficient conditions in terms of $r$ and $d$ for the contractibility or $k$-connectedness of $\mrd$. 
\end{probs}

It seems tempting to conjecture that $\mrd$ is contractible whenever $r>d$, which would imply our main Conjecture, but we have very little evidence to support this. We do however have the following generalization of Theorem \ref{join-starshaped}.

\begin{theorem}
  If $r > d(d+1)$, then $\mrd$ is starshaped.
\end{theorem}

\begin{proof} 
The proof is identical to the proof of Theorem \ref{join-starshaped}. Suppose the rank of $M$ is greater than $d(d+1)$ and let $T$ be an independent set of size $d(d+1)+1$. By Tverberg's theorem there exists a partition $T = T_1 \cup \cdots \cup T_{d+1}$ and a point $t\in \R^d$ such that $t \in \conv T_i$ for every $i$. We will show that every point $x\in M_r^d$ is visible from $t$.  

It suffices to consider the case when $x$ belongs to the boundary of $\mrd$, so by Carath\'eodory's theorem we may assume $x$ belongs to an independent simplex of dimension at most $d-1$. Thus $x$ is contained in the convex hull of an independent set $Y$ with $|Y| \leq d$. By repeated application of the independence augmentation axiom, there is a subset $S\subset T$ such that $S\cap Y = \emptyset$ and $\mbox{rk}(S\cup Y) = |S\cup Y| = d(d+1)+1$, that is, $S\cup Y$ is independent. Therefore $|S|> d^2$, and by the pigeon-hole principle there exists some $T_j\subset S$.  This implies that the closed segment $[t,x]$ is contained in $\conv (S\cup Y)$ which is contained in $\mrd$. 
\end{proof}

We also have the following generalization of Theorem \ref{simp-conn}.

\begin{theorem}
  If $r>\frac{d+2}{2}$, then $\mrd$ is simply connected.
\end{theorem}

\begin{proof}
The same proof as the one for Theorem \ref{simp-conn} works here. We think of $\mrd$ as the image of a map $f:Y \to \R^d$ where $Y$ is the independence complex of $M$ and $f$ is linear on every simplex of $Y$. The case $d=1$ is trivial since the independence complex of a matroid of rank $r\geq 2$ is connected, so we assume $d\geq 2$.  Let $\sigma_1$ and $\sigma_2$ be two $(r-1)$-simplices of $Y$ such that $\sigma_1\cap \sigma_2 = \emptyset$ and $f(\sigma_1) \cap f(\sigma_2) \neq \emptyset$. Then there is a proper face $\tau$ of either $\sigma_1$ or $\sigma_2$ such that $f(\sigma_1)\cap f(\tau) \cap f(\sigma_2) \neq \emptyset$, and by the independence augmentation axiom, $\tau$ is contained in an $(r-1)$-simplex $\tau_0\in Y$ such that $\sigma_1\cap \tau_0 \neq \emptyset$, $\sigma_2\cap \tau_0 \neq \emptyset$, and $f(\sigma_1)\cap f(\tau_0) \cap f(\sigma_2) \neq\emptyset$. Therefore the same argument as in the proof of Theorem \ref{simp-conn} shows that any path in $\mrd$ is homotopic to the image of some path in $Y$. The result now follows from the fact that the independence complex of a matroid of rank $r\geq 3$ is simply connected. 
\end{proof}

\begin{remark}
It is natural to ask whether there are other reasonable classes of simplicial complexes for which our main Problem might yield interesting results. 
\end{remark}

\section{Acknowledgments}
The authors are grateful to two anonymous referees for helpful comments and suggestions.
I.~B. was partially supported by ERC  Advanced Research Grant no 267165 (DISCONV), and by Hungarian National Research Grant K 83767.
A.~F.~H. was supported by the Basic Science Research Program through the National Research Foundation of Korea funded by the Ministry of Education, Science and Technology (NRF-2010-0021048).
R.~K. was supported by the Dynasty foundation.

\end{document}